\documentclass[11pt,a4paper,english]{amsart}
\usepackage{babel}
\usepackage[latin1]{inputenc}
\usepackage{amsmath}
\usepackage{amsthm}
\usepackage{amsfonts}
\usepackage{indentfirst}
\usepackage{graphicx}

\usepackage{amssymb}
\usepackage[mathscr]{eucal}

\newtheorem{de}{Definition}
\newtheorem{pro}{Proposition}
\newtheorem{teo}{Theorem}
\newtheorem{rem}{Remark}
\newtheorem{lem}{Lemma}
\newtheorem{conj}{Conjecture}

\newcommand{\co}{{\mathcal O}}

\newcommand{\gp}{\mathbb{P}}

\newcommand{\gz}{\mathbb{Z}}

\newcommand{\gr}{\mathbb{R}}
\newcommand{\gc}{\mathbb{C}}

\newcommand{\gn}{\mathbb{N}}

\newcommand{\ck}{\mathcal K}

\newcommand{\cf}{{\mathcal F}}

\newcommand{\findemo}{$\ \ \square$}

\renewcommand{\int}{{\rm int}}

\newcommand{\pic}{{\rm Pic}}

\title{Lins Neto's examples of foliations and the Mori cone of blow-ups of $\gp^2$}
\author{F.  Monserrat}
\address{Departamento de Matemática Aplicada, Universidad Politécnica de Valencia\\ Camino de Vera s/n\\ 46022 Valencia, Spain}
\email{framonde@mat.upv.es}

\thanks{Supported by Spain Ministry of Education
 MTM2007-64704}
\date{}

\begin{document}

\maketitle

\begin{abstract}
We use a family of algebraic foliations given by A. Lins Neto to
provide new evidences to a conjecture, related to the
Harbourne-Hirschowitz's one and implying the Nagata's conjecture,
which concerns the structure of the Mori cone of blow-ups of $\gp^2$
at very general points.

\end{abstract}

\section{Introduction}

The Harbourne-Hirschowitz conjecture predicts the solution to the
problem of determining the dimension of every linear system of
curves of $\gp^2$ (the projective plane over an algebraically
closed field, which we shall assume to be $\gc$) with assigned
multiplicities at general points. This conjecture goes back to
Segre \cite{seg} and it has been reformulated by several authors
\cite{harb1, gim1, hir}. We are interested in the following weaker
form \cite{harb2}:

\begin{conj}\label{conj1}
Let $X$ be the blow-up of $\gp^2$ at a finite set of points in
very general position. Then, every integral curve $C$ on $X$ with
negative self-intersection is a $(-1)$-curve of $X$ (that is,
smooth, rational and such that $C^2=-1$).
\end{conj}

Recall that a property is satisfied for $n$ closed points of $\gp^2$
in {\it very general position} if it holds for all $n$-uples
$(p_1,p_2,\ldots,p_n)$ belonging to the complement of a countable
union of proper closed subvarieties of $(\gp^2)^n$.

It is known that the statement of Conjecture \ref{conj1} is
satisfied by the rational curves and also by any curve whose image
on $\gp^2$ has a singularity of multiplicity 2 at one of the
centers of the blow-up \cite{defernex}.

There are several results giving evidence to the
Harbourne-Hirschowitz conjecture (\cite{ah}, \cite{harb3},
\cite{cil}, \cite{cil2}, \cite{ev}, \cite{yang}, \cite{dum} and
\cite{mons} among many others). In the same spirit, the objective of
this note will be to provide evidences to Conjecture \ref{conj1}. To
do that, we shall consider other equivalent formulation (Conjecture
\ref{conj2}) that predicts the structure of the closure of the Mori
cone, $\overline{NE}(X)$, associated with a blow-up $X$ of $\gp^2$
at very general points. In a more precise form, it states that any
generator of extremal ray of $\overline{NE}(X)$ having non-negative
intersection product with the canonical class $K_X$ of $X$ has null
self-intersection. It is known that, when the number $n$ of blown-up
points is greater than 1, the extremal rays of $\overline{NE}(X)$ in
the half-space $(K_X\cdot z<0)$ are exactly those generated by the
classes of the $(-1)$-curves and when $n\geq 9$ there are infinitely
many of them \cite{nag}. When $n\leq 9$ the conjecture is true (see
Remark \ref{bbb}) but very little is known on the intersection
$\overline{NE}(X) \cap (K_X\cdot z\geq 0)$ when $n>9$.

An interesting aspect is that Conjecture \ref{conj1} implies
Nagata's conjecture \cite{nag1} and, therefore, there is a
connection with the symplectic packing problem in dimension four
\cite{mac}.

In section \ref{4}, we show an explicit family of smooth rational
projective surfaces $X$ satisfying that the set of faces of the cone
$NE(X)$ meeting the region $(K_X\cdot z=0)$ (resp., $(K_X\cdot z>
0)$) is not finite (see Proposition \ref{teo1}). Under the
assumption that $X$ is the blow-up of $\gp^2$ at a set of $n\geq 12$
(resp., $n\geq 37$) points in very general position, we prove, also
in Section \ref{4}, a behavior of the cone $NE(X)$ which agree with
Conjecture \ref{conj2}. More specifically, on the one hand Theorem
\ref{gordo} (see also Remark \ref{remgordo}) shows the existence of
infinitely many rays which are contained in the boundary of
$\overline{NE}(X)$ and are generated by elements in the region
$(K_X\cdot z=0)$ (resp., $(K_X\cdot z> 0)$) with null
self-intersection; on the other hand Theorem \ref{gordo2} shows
that, when $n\geq 37$, the set of the above mentioned rays which are
in $\partial \overline{NE}(X)\cap (K_X\cdot z> 0)$ correspond to
infinitely many orbits of the action of the Cremona group. The
proofs of these theorems use the examples of one-parametric families
of algebraically integrable plane foliations provided by Lins Neto
in \cite{l-n} in relation with the so-called Poincaré and Painlevé
problems. For this reason we devote Section \ref{3} to summarize the
necessary background on foliations, to state the results of
\cite{l-n} that we shall use and to prove a key fact in our
development (Proposition \ref{aaaa}). In Section \ref{2} we briefly
summarize basic definitions and facts on the Mori cone and we use
them to state Conjecture \ref{conj2}.

We notice that, although the Lins Neto's families of foliations
provide negative answers with respect to the Poincaré and Painlevé
problems, in this paper they have revealed to be useful to show a
positive result with respect to the Harbourne-Hirschowitz
conjecture.

\section{The Mori cone and the conjecture}\label{2}

Let $X$ be a smooth projective surface and let
$A(X):=(\pic(X)/\equiv)\otimes \gr$, where $\equiv$ denotes
numerical equivalence. $A(X)$ is a real vector space whose
dimension is $\rho(X):={\rm rk}\; \pic(X)$. We shall assume that
$\rho(X)\geq 3$. The {\it Mori cone} of $X$ (also called
Kleiman-Mori cone or cone of curves), which we shall denote by
$NE(X)$, is defined to be the convex cone of $A(X)$ generated by
the images of the effective classes in $\pic(X)$; its closure with
respect to the real topology will be denoted by
$\overline{NE}(X)$. The $\gz$-bilinear form $\pic(X)\times
\pic(X)\rightarrow \gz$ given by Intersection Theory induces a
non-degenerated $\gr$-bilinear pairing $A(X) \times A(X)
\rightarrow \gr$. For each pair $(x,y)\in A(X)\times A(X)$,
$x\cdot y$ will denote its image by this bilinear form, for each
divisor $D$ on $X$, $[D]$ will be its image in $A(X)$ and, for
each real number $\alpha$, $[D]_{>\alpha}$ (resp., $[D]_{\geq
\alpha}$, $[D]_{<\alpha}$, $[D]_{\leq \alpha}$, $[D]^{\perp}$)
will denote the set of those $x\in A(X)$ such that $[D]\cdot
x>\alpha$ (resp., $\geq \alpha$, $<\alpha$,  $\leq \alpha$, $=0$).

Recall that, if $C$ is a convex cone of $A(X)$, a {\it face} of
$C$ is a sub-cone $F\subseteq C$ such that $a+b\in F$ implies that
$a,b\in F$, for all pair of elements $a,b\in C$. The
$1$-dimensional faces of $C$ are also called {\it extremal rays}
of $C$.

Fix an ample divisor $H$ on $X$. By Kleiman's ampleness criterion,
$[H]\cdot x>0$ for all $x\in \overline{NE}(X)\setminus \{0\}$ and,
hence, the cone $\overline{NE}(X)$ is strongly convex. This
implies that it is generated by its extremal rays. Consider the
cone
$$Q(X)=\{x\in A(X)\mid x^2\geq 0,\;\; [H]\cdot x\geq 0\}.$$
One has that $Q(X)\subseteq \overline{NE}(X)$
\cite[II.4.12.1]{kollar1} and, therefore, the extremal rays of
$\overline{NE}(X)$ must be spanned by elements $x\in A(X)$ such that
$x^2\leq 0$. Moreover, the extremal rays of $\overline{NE}(X)$ which
are not in $Q(X)$ are spanned by classes of integral curves $C$ with
$C^2<0$ \cite[II.4.12.3]{kollar1}. As a consequence of the Mori cone
theorem (see \cite[III.1]{kollar1} for instance) the extremal rays
of $\overline{NE}(X)$ meeting the region $[K_X]_{<0}$ are exactly
those spanned by the images in $A(X)$ of the $(-1)$-curves.
Moreover, if $C$ is an integral curve on $X$ such that $C^2<0$ then
$[C]$ generates an extremal ray of $\overline{NE}(X)$
\cite[II.4.12.2]{kollar1}. These considerations allow us to state
the following equivalent formulation of Conjecture \ref{conj1}
(which holds trivially when $\rho(X)<3$):

\begin{conj}\label{conj2}
If $X$ is the blow-up of $\gp^2$ at a finite set of points in very
general position then the extremal rays of the cone
$\overline{NE}(X)$ in the region $[K_X]_{\geq 0}$ are contained in
$\partial Q(X)$.
\end{conj}

\begin{rem}\label{bbb}

{\rm Notice that Conjecture \ref{conj2} is true when $\rho(X)\leq
10$; in fact, the equality $\overline{NE}(X)\cap [K_X]_{\geq 0}=
Q(X)\cap [K_X]_{\geq 0}$ holds in this case. Indeed, if $\rho(X)\leq
9$ then $Q(X)\cap [K_X]_{\geq 0}= \overline{NE}(X)\cap [K_X]_{\geq
0}=\{0\}$ because $-K_X$ is ample. If $\rho(X)=10$ then one has that
$\overline{NE}(X)\cap [K_X]_{\geq 0}\subseteq [K_X]^{\perp}\cap
Q(X)=Q(X)\cap [K_X]_{\geq 0}$, where the inclusion holds because
$-K_X$ is nef and there are not integral curves with negative
self-intersection whose images belong to $[K_X]^{\perp}$, and the
equality follows from the proof of \cite[Cor. 1.ii]{galmon}.

}
\end{rem}

\begin{rem}
{\rm Conjecture \ref{conj1} is not equivalent to the fact that the
equality $\overline{NE}(X)\cap [K_X]_{\geq 0}= Q(X)\cap [K_X]_{\geq
0}$ holds for whichever surface $X$ obtained by blowing-up $\gp^2$
at a finite set of points in very general position. Although a
reformulation of Conjecture 1 in these terms is given in
\cite{defernex}, its author (in a private communication  to me and
in a note which will appear elsewhere) asserts that this
reformulation is only correct when $\rho(X)\leq 11$; in fact, he
shows that  the inclussion $Q(X)\cap [K_X]_{\geq 0}\subseteq
\overline{NE}(X)\cap [K_X]_{\geq 0}$ is strict whenever $\rho(X)>11$
(independently of any conjecture).

}
\end{rem}

\section{Families of algebraically integrable foliations}\label{3}

With the exceptions of Lemma \ref{lema1} and Proposition
\ref{aaaa}, this section is expository and its aim is, on the one
hand, to summarize some basics facts concerning foliations (see
\cite{brunella} and \cite{G-M}) and, on the other hand, to state
the results of \cite{l-n} that we shall use to obtain the main
results of the paper.

Let $X$ be a smooth projective surface defined over $\gc$ and let
$\Theta_X$ be its associated tangent sheaf. An {\it algebraic
foliation with singularities} (foliation, in the sequel) ${\mathcal
F}$ on $X$ is given by an open covering $\{U_j\}_{j\in J}$ of $X$
and vector fields $v_j\in H^0(U_j,\Theta_X)$ with isolated zeroes
such that $v_i=g_{ij}v_j$ on $U_i\cap U_j$, where $g_{ij}\in
H^0(U_i\cap U_j, \co^*_X)$ for all $i,j\in J$. A closed point $p\in
X$ is a {\it singular point} of $\cf$ if it is a zero of $v_j$ for
some $j\in J$. Notice that the set of singular points is finite.
Given $p\in X$, a {\it separatrix} of $\cf$ at $p$ will be an
irreducible germ $f\in \co^{hol}_{X,p}$ (where $\co_X^{hol}$ denotes
the sheaf of holomorphic functions) such that $v_j(f)$ is a multiple
of $f$, if $p\in U_j$. An {\it algebraic invariant curve} $C$ will
be an integral curve on $X$ such that the irreducible components of
its germ at each point $p\in C$ (viewed as an element of
$\co^{hol}_{X,p}$) are separatrices at $p$.

If $p$ is a singular point of $\cf$ and $p\in U_j$, we shall say
that $p$ is a {\it non-degenerated} singularity if the jacobian
matrix $Dv_j(p)$ is non-singular. In this case, if $\lambda_1$ and
$\lambda_2$ denote the eigenvalues of $Dv_j(p)$, the quotients
$\lambda_1/\lambda_2$ and $\lambda_2/\lambda_1$ are called {\it
characteristic numbers} of the singularity and they are analytic
invariants. A singular point $p$ is called a {\it reduced} (resp.,
{\it dicritical}) singularity if their characteristic numbers are
not positive rational numbers (resp., there exist infinitely many
separatrices passing thorough $p$). If $p$ is a singularity of
$\cf$ such that the separatrices of $\cf$ at $p$ are given by the
levels of a meromorphic function of the type $\frac{u^a}{v^b}$ for
certain local analytic coordinates $(u,v)$, then we say that $\cf$
{\it has a local meromorphic first integral} at $p$ of the type
$\frac{u^a}{v^b}$.

Given a vector field $p(x,y)\frac{\partial}{\partial x} +
q(x,y)\frac{\partial}{\partial y}$, where $p$ and $q$ are
polynomials on $\gc^2$, it can be extended to a unique foliation
$\cf$ of $\gp^2$. The singular points of $\cf$ in the affine chart
$\gc^2$ are the common zeroes of $p$ and $q$. Moreover, there exists
a positive integer $d$ such that the above vector field can be
written in the form
$$a(x,y)\frac{\partial}{\partial x}+b(x,y)\frac{\partial}{\partial
y}+g(x,y)\left(x \frac{\partial}{\partial x}+y
\frac{\partial}{\partial y}\right)$$ where, either $a,b$ are
polynomials of degree at most $d$ and $g$ is a homogeneous
polynomial of degree $d$, or $g\equiv 0$, $\max\{ \deg(a),\deg(b)
\}=d$ and the homogeneous parts of $a$ and $b$ of degree $d$ are
not of the form $x\cdot h$ and $y\cdot h$ respectively. The
integer $d$ is the {\it degree} of $\cf$, and it is also the
number of tangencies of $\cf$ with a generic line, linearly
embedded in $\gp^2$.

A foliation $\cf$ of $\gp^2$ {\it has a rational first integral}
(or it is {\it algebraically integrable}) if there exists a
rational map $R: \gp^2 \cdots \rightarrow \gp^1$ such that the
irreducible components of the closures of its fibers are algebraic
invariant curves of $\cf$. Taking homogeneous coordinates
$[X,Y,Z]$ on $\gp^2$, the map $R$ is defined by two homogeneous
polynomials $F,G\in \gc[X,Y,Z]$ of the same degree $m$ which can
be taken in such a way that general fibers of $R$ are irreducible;
the {\it degree} of the first integral $R$ is defined to be $m$.
Hence, the foliation $\cf$ determines a unique irreducible pencil
(i.e., with irreducible general fibers) of plane curves ${\mathcal
P}_{\cf}:=\langle F,G \rangle \subseteq H^0(\gp^2,\co_{\gp^2}(d))$
given by the levels of the rational function $R=\frac{F}{G}$;
moreover, the integral components of the curves in ${\mathcal
P}_{\cf}$ are exactly all the algebraic invariant curves.

A {\it configuration} over $\gp^2$ will be a finite sequence
${\ck}=(p_1,p_2,\ldots,p_n)$ of closed points such that $p_1$
belongs to $X_1:=\gp^2$ and, inductively, if $i\geq 1$ then $p_i$
belongs to the blow-up $X_i$ of $X_{i-1}$ at $p_{i-1}$. Also, we
shall denote by $\pi_{\ck}:Z_{\ck} \rightarrow \gp^2$ the morphism
given by the composition of all the successive blow-ups centered
at the points of $\ck$.

Consider a non-degenerated foliation $\cf$ (that is, a foliation
whose singularities are non-degenerated) of $\gp^2$.  Seidenberg's
result of reduction of singularities \cite{seid} proves the
existence of a sequence of blow-ups $X_{n+1} \mathop
{\longrightarrow} \limits^{\pi _{n} } X_{n} \mathop
{\longrightarrow} \limits^{\pi _{n-1} }  \cdots \mathop
{\longrightarrow} \limits^{\pi _2 } X_2 \mathop {\longrightarrow}
\limits^{\pi _1 } X_1 : = \gp^2$ and foliations $\cf_i$ on $X_i$
($\cf_1=\cf$ and the remaining ones are successive transforms of $
\cf$) such that $\cf_{n+1}$ has only reduced singularities. If, in
addition, $\cf$ has a rational first integral, $R$, then the
non-reduced singularities are exactly the dicritical ones;
moreover, elementary calculations using the local equations of the
blow-up show that the Seidenberg's reduction process coincides
with the minimal composition of point blow-ups $\pi_{\cf}: Z_{\cf}
\rightarrow \gp^2$ eliminating the indeterminacies of $R$. The
morphism $\pi_{\cf}$ and the configuration given by the sequence
of centers of the blow-ups used to get it, which we shall denote
by ${\mathcal B}_{\cf}$, are essentially unique because different
admissible (in the obvious way) arrangements of the points give
rise to $\gp^2$-isomorphic surfaces. ${\mathcal B}_{\cf}$  is
called {\it configuration of base points} of ${\mathcal P}_{\cf}$.

\begin{de}\label{def1}
{\rm  A {\it one-parametric family of foliations} of $\gp^2$  will
be a set $\{\cf_{\alpha}\}_{\alpha \in U}$, where $U$ is a
connected open subset of $\gc$ and $\cf_{\alpha}$ are foliations
of $\gp^2$ which extend polynomial vector fields
$a_{\alpha}(x,y)\frac{\partial}{\partial
x}+b_{\alpha}(x,y)\frac{\partial}{\partial y}$ on $\gc^2$ such
that the coefficients of  $a_{\alpha}(x,y)$ and $b_{\alpha}(x,y)$
are functions of $\alpha$ which are holomorphic in $U$. }
\end{de}

In \cite{l-n}, Lins Neto defines, for all integers $d\geq 2$, a
one-parametric familiy of foliations of degree $d$, $\Upsilon^d=\{
\cf^d_{\alpha} \}_{\alpha \in \gc \setminus A^d}$, where $A^d$ is a
finite subset of $\gc$. For all $d\geq 2$, the family $\Upsilon^{d}$
has non-degenerated singularities of fixed analytic type \cite[Def.
1]{l-n} and all the foliations in the family have the same
dicritical singularities. Moreover, there exists a dense countable
subset $E^d\subseteq \gc\setminus A^d$ such that $\cf_{\alpha}^d$
has a rational first integral for all $\alpha \in E^d$. If $d\in
\{2,3,4\}$ and $\alpha\in E^d$, then the general algebraic invariant
curves of $\cf_{\alpha}^d$ are elliptic curves. If $d\geq 5$, then
the following property is satisfied: for any $k>0$ the set $\{\alpha
\in E^d \mid \mbox{ the genus of a general algebraic invariant curve
by $\cf^d_{\alpha}$ is }\leq k\}$ is finite. We point out here that, althouth
the existence of the sets $E^d$ is proved,
they are not explicitly described in \cite{l-n}. Next, we shall
summarize other properties of these families that will be of
interest for us (see \cite{l-n} for complete details).

With respect to the family $\Upsilon^4$, there are 9 lines which are
common invariant curves of all the foliations of the family and the
dicritical singularities are the 12 points of intersection among
these lines (see \cite[Fig. 1]{l-n}). Moreover, they are of radial
type, that is, they have a local meromorphic first integral of the
type $u/v$ for certain local coordinates $(u,v)$.

The family $\Upsilon^3=\{\cf_{\alpha}^3\}_{\alpha \in
\mathbb{C}\setminus \{0,1,j,j^2\}}$ is such that
$\cf_{\alpha}^3={\mathcal G}_{\alpha}$, where ${\mathcal
G}_{\alpha}$ is the foliation satisfying the following property:
$\cf_{\alpha}^4$ is the pull-back of ${\mathcal G}_{\alpha}$ by the
map $T(x,y)=(x+y,x\cdot y)$ \cite[Sect. 2.3]{l-n}. The foliations in
$\Upsilon^3$ have 5 common invariant curves: 2 conics and 3 lines.
The 8 dicritical singularities are the points of intersection among
theses curves \cite[Fig. 4]{l-n}. Three of them are of radial type
and the remaining ones have a meromorphic first integral of the type
$u^2/v$.

The foliations in $\Upsilon^2$ are obtained from those in
$\Upsilon^3$ by a Cremona transformation \cite[Sect. 2.4]{l-n}. All
of them have 2 common invariant curves: a quartic $Q$ and a line
$R$. There are 5 dicritical singularities. Two of them (say $M$ and
$N$) are smooth points of $Q$ and are points of tangency of $Q$ and
$R$; they have a meromorhic first integral of the type $u^2/v$. The
remaining ones (say $J$, $K$ and $L$) are cuspidal points of $Q$
(see \cite[Fig. 7]{l-n}) and have a meromorphic first integral of
the type $u^3/v^2$.

The foliations in $\Upsilon^{3r-1}$ (with $r\geq 2$) are obtained by
pulling-back those in $\Upsilon^2$ by certain rational maps
$F:\gp^2\rightarrow \gp^2$ (see \cite[Sect. 3.1]{l-n}). The
dicritical singularities are those in $F^{-1}(\{J,K,L,M,N\})$. $J,
K$ and $L$ are not critical values of $F$ and $F^{-1}(\{J,K,L\})$
contains $3r^2$ singularities with meromorphic first integral of the
type $u^3/v^2$. $M$ and $N$ are critical values of $F$ and
$F^{-1}(\{M,N\})$ contains $2r$ singularities with meromorphic first
integral of the type $u^r/v^2$.

In the same way as above, the foliations in $\Upsilon^{3r}$ (with
$r\geq 2$) are obtained by pulling-back those in the family
$\Upsilon^3$ by certain rational maps $\gp^2\rightarrow \gp^2$ (see
\cite[Sect. 3.2]{l-n}). The dicritical singularities are the
pre-images of the dicritical singularities of the foliations in
$\Upsilon^3$. The analytic types of these singularities are the
following ones: one radial singularity, $3r^2$ singularities with
local meromorphic first integral of the type $u^2/v$, $2r$
singularities with local meromorphic first integral of the type
$u^r/v$ and $2r$ more with meromorphic first integral of the type
$u^r/v^2.$

The foliations in $\Upsilon^{3r+1}$ (with $r\geq 2$) are obtained
also by pulling-back those in $\Upsilon^4$ by certain rational maps
$\gp^2\rightarrow \gp^2$. The dicritical singularities are the
pre-images of those of the foliations in $\Upsilon^4$ and they have
the following analytic types: $3r^2+3$ radial singularities and $6r$
singularities with local meromorphic first integral of the type
$u^r/v$.

\begin{rem}\label{nuevanota}
{\rm Although all the above foliations $\cf^d_{\alpha}$ have local
meromorphic first integrals at their dicritical singular points,
only those corresponding to indices $\alpha$ in $E^d$ have a
rational first integral. We notice that, for $\alpha\in E^d$, the
knowledge of the type of the local meromorphic first integrals of
$\cf_{\alpha}^d$ is far to provide sufficient information to recover
the rational first integral of $\cf_{\alpha}^d$, even the
singularity types of the generic algebraic invariant curves at the
dicritical singular points. Indeed, if $\cf^d_{\alpha}$ (with
$\alpha\in E^d$)  has a local meromorphic first integral of the type
$u^{\rho}/v^{\delta}$ (with $\rho$ and $\delta$ relatively primes)
at a dicritical point $p$ we are saying that the germs at $p$ of the
curves of the pencil ${\mathcal P}_{\cf^d_{\alpha}}$ are
$s(\lambda_1 P(u^{\rho},v^{\delta})+\lambda_2
Q(u^{\rho},v^{\delta}))$, $(\lambda_1,\lambda_2)\in
\mathbb{C}^2\setminus \{(0,0)\}$, where $s$ is a unit of
$\co_{\gp^2,p}^{hol}$ and $P$ and $Q$ are homogeneous polynomials of
the same degree, say $k_p$ (see \cite[Section 2.9]{julio} for
instance). To determine the singularity types of the generic curves
of the pencil ${\mathcal P}_{\cf^d_{\alpha}}$ one needs to know the
set of values $k_p$, which is essentially an equivalent datum to the
degree of the rational first integral (see \cite[Th. 3.7]{galmon3}).
Poincaré considered, in \cite{poi2, poi3}, the following classical
problem: to obtain a bound of such a degree in terms of the degree
of the foliation. It is well-known that, in general, it is not
possible to find that bound (even if the analytic types of the
singularities of the foliation are given, as Lins Neto proves in
\cite{l-n} using the mentioned families).

}
\end{rem}

Now, we shall see that all the foliations of a given one-parametric
family $\Upsilon^d$ have a common resolution of singularities. For
this purpose, we shall prove the following

\begin{lem}\label{lema1}
Let $\{\cf_{\alpha}\}_{ \alpha \in U}$ be a one-parametric family
of foliations of $\gp^2$, $U$ being a connected open subset of
$\gc$. Let $p$ be a point of $\gp^2$. Suppose that:
\begin{itemize}
\item[(a)] All the foliations $\cf_{\alpha}$ have a
non-degenerated singularity at $p$ of the same analytic type.

\item[(b)] The characteristic numbers of $\cf_{\alpha}$, $\alpha
\in U$, at $p$ are rational and positive, say $a/b$ and $b/a$,
where $a,b\in \gn$ are relatively primes.

\item[(c)] All the foliations $\cf_{\alpha}$ have two common
separatrices through $p$.
\end{itemize}
Then,  the configuration of infinitely near points involved in the
resolution of the singularity of $\cf_{\alpha}$ at $p$ is constant
for all $\alpha\in U$ (that is, the singularities at $p$ of all
the foliations in the family  have a common resolution).
\end{lem}

\noindent {\it Proof.}  Let ${\mathcal C}_{\alpha}=(p_1(\alpha)=p,
p_2(\alpha)\ldots, p_s(\alpha))$, $\alpha \in U$, be the
configuration of those infinitely near points involved in the
resolution of the singularity of $\cf_{\alpha}$ at $p$. The result
is trivial when $s=1$, that is, the singularity is of radial type.
So, we shall assume that either $a$ or $b$ is greater than 1.

Applying \cite[Lem. 1]{l-n} one has that, for each $\alpha \in U$,
there exists a holomorphic coordinate system
$(W_{\alpha},(u_{\alpha},v_{\alpha}))$ with $p\in W_{\alpha}$,
$u_{\alpha}(p)=v_{\alpha}(p)=0$, such that
$\frac{u_{\alpha}^a}{v_{\alpha}^b}$ is a meromorphic first
integral of $\cf_{\alpha}$ in a neighborhood of $p$. Therefore,
for each $\alpha \in U$, the local analytic separatrices of
$\cf_{\alpha}$ through $p$ are the irreducible components of the
analytic germs in the local linear pencil
$\Delta_{\alpha}:=\{\lambda u_{\alpha}^a+\mu v_{\alpha}^b=0 \mid
(\lambda,\mu)\in \gc^2\setminus \{(0,0)\}\}$. Moreover, since
${\mathcal C}_{\alpha}$ is also the configuration of base points
of $\Delta_{\alpha}$, it is clear that  $p_j(\alpha)$ belongs to
the exceptional divisor created by blowing-up $p_{j-1}(\alpha)$
(for all $j\geq 2$) and that, in order to prove the equalities
${\mathcal C}_{\alpha}={\mathcal C}_{\beta}$ for all $\alpha,
\beta\in U$, it suffices to show that all the points $p_2(\alpha)$
coincide. Now, notice that all the germs in $\Delta_{\alpha}$ are
irreducible and have the same tangent direction, except one of
them (defined by $v_{\alpha}^b=0$ if the inequality $a<b$ is
assumed). Hence, the strict transform in the blow-up at $p$ of one
of the two common separatrices given in (c) must pass through
$p_{2}(\alpha)$. Therefore there are, at most, two possibilities
(which do not depend on the value of $\alpha$) for the point
$p_{2}(\alpha)$, say $e_1$ and $e_2$, given by the tangent
directions defined by the common separatrices. If $e_1=e_2$, then
the result follows. So, we shall assume that $e_1$ and $e_2$ are
different points of the first exceptional divisor $E$. Then, for
each $\alpha \in U$, $p_2(\alpha)$ coincides with one of these
points, and the remaining one is a reduced singularity of the
transform of $\cf_{\alpha}$.

Observe that, on the one hand, for each $i=1,2$, there exists a
holomorphic function $f_i$ on $U$ such that the characteristic
numbers of the transform of the foliation $\cf_{\alpha}$ at the
point $e_i$ are $f_i(\alpha)$ and $1/f_i(\alpha)$ (both functions
$f_i$ and $1/f_i$ are defined in $U$ because we are assuming that
the singularities are non-degenerated). On the other hand, for
each $\alpha \in U$, $p_2(\alpha)$ is a dicritical singularity of
the transform of $\cf_{\alpha}$ and its characteristic numbers are
$\frac{a}{b-a}$ and $\frac{b-a}{a}$ (assuming that $a<b$). Now,
consider the following holomorphic functions defined in $U$:
$g_i(\alpha):=(f_i(\alpha)-\frac{a}{b-a}) \cdot
(f_i(\alpha)-\frac{b-a}{a}), \;\;i=1,2$. Notice that $g_1(\alpha)
g_2(\alpha)=0$ and $g_1(\alpha) -g_2(\alpha)\not=0$ for all
$\alpha\in U$. Since both functions are holomorphic, it follows
that one of them (say $g_{i_0}$) is identically zero in $U$.
Therefore $p_2(\alpha)=e_{i_0}$ for all $\alpha \in U$.\findemo

\begin{pro}\label{aaaa}
Let $d\geq 2$ be an integer. Then, the configurations ${\mathcal
B}_{\cf^d_{\alpha}}$ coincide (up to re-arrangement) for all
$\alpha \in \gc\setminus A^d$.
\end{pro}
\noindent {\it Proof.} Taking into account the above description
of the dicritical singularities of the foliations in the families
$\Upsilon^d$  and applying Lemma \ref{lema1}, the result is easily
deduced when $d \not\equiv 2$ mod 3. In the case of $\Upsilon^2$,
it can be checked by inspection that the infinitely near points
involved in the resolution of the singularities at $J$, $K$ and
$L$ are the same for all the foliations in the family. By Lemma
\ref{lema1}, the same is true for $M$ and $N$. Therefore,
${\mathcal B}_{\cf^2_{\alpha}}$ does not depend on $\alpha$. In
the case of $\Upsilon^{3r-1}$ (for $r\geq 2$) the $3r^2$
singularities in $F^{-1}(\{J,K,L\})$ have also the same
resolution; the reason is the following one: for any $p\in
\{J,K,L\}$ and for any $q\in F^{-1}(p)$,  $F$ defines a
biholomorphism in a neighborhood of $q$, since $p$ is not a
critical value of $F$. Again by Lemma \ref{lema1}, the same occurs
for the remaining singularities.\findemo\\

In view of the above proposition, we shall consider that, for a
fixed $d\geq 2$, all the configurations ${\mathcal
B}_{\cf^d_{\alpha}}$ are the same and we shall denote it by
${\mathcal B}^d$. Also, $X^d$ will denote the surface
$Z_{{\mathcal B}^d}$ obtained by blowing-up the points in
${\mathcal B}^d$.

\section{The Mori cone of the blow-ups}\label{4}

The rest of the paper is devoted to prove the two main results. We
shall use the notations of the preceding sections.

In order to state and prove the first of these results, we denote
by $\Delta(C)$ the set of faces of a convex cone $C$ and, for each
integer $d\geq 2$, we consider the function
$$\Psi^d:E^d\rightarrow \Delta(NE(X^d))$$ which maps every
$\alpha\in E^d$ to $[D_{d,\alpha}]^{\perp}\cap NE(X^d)$, where
$D_{d,\alpha}$ denotes the strict transform on $X^d$ of a general
invariant curve by $\cf^d_{\alpha}$ that is, a general curve of
the pencil ${\mathcal P}_{\cf^d_{\alpha}}$ (note that all these
foliations have a rational first integral). This map is
well-defined. Indeed, since the complete linear system $|D_{d,
\alpha}|$ is base-point-free, one has that $[D_{d, \alpha}]^\perp$
is a supporting hyperplane of the cone $NE(X^d)$ and, therefore,
$[D_{d,\alpha}]^{\perp}\cap NE(X^d)$ is a face of $NE(X^d)$.

\begin{pro}\label{teo1}
For each positive integer $d\geq 2$, the map $\Psi^d$ is
injective. Moreover:
\begin{itemize}
\item[(a)] If $d\leq 4$, then $\Psi^d(\alpha)\cap
[K_{X^d}]^{\perp}\not=\{0\}$ for all $\alpha\in E^d$.

\item[(b)] If $d\geq 5$, then for any $k>0$ the set $\{\alpha \in
E^d \mid \Psi^d(\alpha)\subseteq [K_{X^d}]_{\leq k} \}$ is finite.
\end{itemize}
In particular, the set of faces of the cone $NE(X^d)$ meeting the
region $[K_{X^d}]^{\perp}$ is not finite and, if $d\geq 5$, the
same happens for the set of faces of $NE(X^d)$ meeting the region
$[K_{X^d}]_{>0}$.
\end{pro}

\begin{proof}

Set any integer $d\geq 2$. By applying Bézout's theorem to two
general curves of ${\mathcal P}_{\cf^d_{\alpha}}$, it is easy to
deduce that, for each $\alpha\in E^d$, $D_{d, \alpha}^2=0$ and,
therefore, $[D_{d, \alpha}]$ belongs to $\Psi^d(\alpha)$. Now, two
different values $\alpha, \beta \in E^d$ give rise to two different
faces $\Psi^d(\alpha)$ and $\Psi^d(\beta)$. Indeed, if the two faces
were the same then both foliations $\cf_{\alpha}^d$ and
$\cf_{\beta}^d$ would have the same invariant curves \cite[Th.
1]{galmon2} and, hence, they would coincide; this is a
contradiction. So, the map $\Psi^d$ is injective.

Clauses (a) and (b) follow easily taking into account the paragraph
after Definition \ref{def1} and the Adjunction Formula applied to
the divisors $D_{d,\alpha}$.

\end{proof}

\begin{rem}\label{cardinal}
{\rm In Section \ref{3}, it is given the local analytic type of the
first integrals of the foliations of each family at each dicritical
singularity. This allows to compute, for each $d\geq 2$, the number
of points $\ell({\mathcal B}^d)$ involved in the configuration
${\mathcal B}^d$ (see Remark \ref{nuevanota}). In fact, one has that
$\ell({\mathcal B}^2)=13$, $\ell({\mathcal B}^3)=13$,
$\ell({\mathcal B}^4)=12$, $\ell({\mathcal B}^{3n+1})=9n^2+3$,
$\ell({\mathcal B}^{3n-1})=10n^2+3(1-(-1)^n)n/2$ and $\ell({\mathcal
B}^{3n})=9n^2+1+3(1-(-1)^n)n/2$ for all $n\geq 2$. }
\end{rem}

Let $\ck=(p_1,p_2,\ldots,p_n)$ be an arbitrary configuration over
$\gp^2$. Each blow-up at $p_i$ gives rise to an exceptional divisor
$E_i$ whose total (resp., strict) transform on $Z_{\ck}$ will be
denoted by $E_i^\ck$ (resp., $\tilde{E}^\ck_i$). In the same way,
for each effective divisor $C$ on $X$, $C^\ck$ (resp.,
$\tilde{C}^\ck$) will be the total (resp., strict) transform of $C$
on $Z_{\ck}$. The system $\{[L^\ck], [E_0^\ck],[E_1^\ck],\ldots,
[E_n^\ck]\}$ is a basis of the vector space $A(Z_{\ck})$, $L$
denoting a general line on $\gp^2$.

For each positive integer $n$, there exists a smooth projective
variety $Y_{n-1}$ whose closed points are naturally identified
with the configurations over $\gp^2$ with $n$ points. These
varieties, known as {\it iterated blow-ups}, were introduced by
Kleiman in \cite{kle1} and \cite{kle2} and they have also been
treated in \cite{harb2-1}. There is a family of smooth projective
morphisms $Y_{n}\rightarrow Y_{n-1}$ and relative divisors
$F_0,F_1,\ldots,F_n$ on $Y_{n}$ such that the fiber over a given
configuration $\ck=(p_1,\ldots,p_n)$ (viewed as a point of
$Y_{n-1}$) is isomorphic to the surface $Z_{\ck}$ obtained by
blowing-up the points in $\ck$ and, if $i\geq 1$ (resp., $i=0$),
the restriction of $F_i$ to this fiber corresponds to the total
transform $E_i^{\ck}$ of the exceptional divisor appearing in the
blow-up centered at $p_i$ (resp., the total transform of a general
line of $\gp^2$)\cite[Prop. I.2]{harb2-1}. For each positive
integer $d$ and for each sequence of non-negative integers
$m_1,\ldots,m_n$ we apply the Semicontinuity Theorem
\cite[III.12.8]{har} to the invertible sheaf
$\co_{Y_{n}}(dF_{0}-m_1F_1-m_2F_2-\ldots-m_nF_n)$, obtaining that
the function $Y_{n-1}\rightarrow \gz$ given by ${\ck} \mapsto
h^0(Z_{\ck}, \co_{Z_{\ck}}(dL^{\ck}-\sum_i m_iE^{\ck}_i))$ is
upper-semicontinuous. Moreover, the subset $U\subseteq Y_{n-1}$
given by the configurations $(p_1,p_2,\ldots,p_n)$ such that
$p_i\in \gp^2$ for all $i=1,2,\ldots,n$ is dense in $Y_n$ (see
\cite{kle1}). Then, denoting by $\ck_0(n)$ a configuration whose
elements are $n$ points of $\gp^2$ in very general position, one
has that
\begin{equation}\label{qqq}
h^0(Z_{\ck}, \co_{Z_{\ck}}(dL^{\ck}-\sum_i m_iE^{\ck}_i))\geq
h^0(Z_{\ck_0(n)}, \co_{Z_{\ck_o(n)}}(dL^{\ck_0(n)}-\sum_i
m_iE^{\ck_0(n)}_i))
\end{equation}
for all triplets $(d,\ck, \{m_i\}_{i=1}^n)$ such that $d$ is a
positive integer, $\ck$ is a configuration with $n$ points and
$\{m_i\}_{i=1}^n$ is a sequence of non-negative integers.

Let $n$ be a positive integer and set $\ck_0(n)=(p_1,
p_2,\ldots,p_n)$. For each configuration $\mathcal
C=(p_1',p_2',\ldots,p_m')$ such that $m\leq n$, the map
$A(Z_{\mathcal C})\rightarrow A(Z_{\ck_0(n)})$ given by
$[L^{\mathcal C}]\mapsto [L^{\mathcal \ck_0(n)}]$ and
$[E_i^{\mathcal C}]\mapsto [E_i^{\mathcal \ck_0(n)}]$ for
$i=1,2,\ldots,m$, is a monomorphism of vector spaces; then,
identifying $A(Z_{\mathcal C})$ with its image, we can assume an
inclusion $A(Z_{\mathcal C})\subseteq A(Z_{\ck_0(n)})$. We shall
use this identification in the rest of the paper. Also, if
$x=\lambda_0[L^{\mathcal \ck_0(n)}]+ \sum_{i=1}^n \lambda_i
[E_i^{\mathcal \ck_0(n)}]\in A(Z_{\ck_0(n)})$, with $\lambda_i\in
\gr$ for all $i$, we shall denote by $N(x)$ the cardinality of the
set $\{i\mid 1\leq i\leq n \mbox{ and } \lambda_i\not=0\}$.

Finally, we shall prove the two results (Theorems \ref{gordo} and
\ref{gordo2}) that provide an evidence to Conjecture \ref{conj2}.

\begin{teo}\label{gordo}
Let $X_n$ be the surface $Z_{\ck_0(n)}$ obtained by blowing-up
$n\geq 12$ points of $\gp^2$ in very general position. Consider the
sets $S(n):=\{d\in \gz \mid d\geq 2 \mbox{ and } \ell({\mathcal
B}^d)\leq n\}$ and $E(n):=\{(d,\alpha)\mid d\in S(n) \mbox{ and }
\alpha \in E^d\}$.
\begin{itemize}

\item[(a)] $[D_{d,\alpha}]\in \partial \overline{NE}(X_n)\cap
\partial Q(X_n)$ for all $(d,\alpha)\in E(n)$.

\item[(b)] If $(d_1,\alpha_1),(d_2,\alpha_2)\in E(n)$, then
$\gr_{>0}[D_{d_1,\alpha_1}]\not=\gr_{>0} [D_{d_2,\alpha_2}]$
whenever either $d_1=d_2$ and $\alpha_1\not=\alpha_2$, or
$d_1\not=d_2$ and $\{d_1,d_2\}\not=\{2,3\}$.

\item[(c)]  If $d\in S(n)\cap \{2,3,4\}$, then $[D_{d,\alpha}]\in
[K_{X_n}]^{\perp}$ for all $\alpha \in E^d$.

\item[(d)] For each $k>0$, the set $\{\alpha\in E^d \mid
[D_{d,\alpha}]\in [K_{X_n}]_{\leq k}\}$ is finite whenever $d\in
S(n)\setminus \{2,3,4\}$.

\end{itemize}
In particular, for all $n\geq 12$ (resp., $\geq 37$), the set of
rays contained in  $\partial \overline{NE}(X_n)\cap
\partial Q(X_n)\cap [K_{X_n}]^{\perp}$ (resp.,  $\partial
\overline{NE}(X_n)\cap
\partial Q(X_n)\cap [K_{X_n}]_{>0}$) is not finite, and the cone
$\overline{NE}(X_n)$ has infinitely many faces meeting the region
$[K_{X_n}]^{\perp}$ (resp., $[K_{X_n}]_{>0}$).

\end{teo}

\begin{proof}

For each $d\in S(n)$ consider a configuration ${\mathcal
C}^d:=(p_1,p_2,\ldots,p_r,q_1,q_2,\ldots,q_{n-r})$, where ${\mathcal
B}^d=(p_1,p_2,\ldots,p_r)$ and $q_1,q_2,\ldots,q_{n-r}$ are
different closed points of $\gp^2$ which are not in ${\mathcal
B}^d$. For every pair $(d,\alpha)\in E(n)$ one has that
$D_{d,\alpha}^2=0$ and, therefore, $[D_{d,\alpha}]$ belongs to the
boundary of $Q(Z_{{\mathcal C}^d})=Q({X_n})$; so, it also belongs to
$\partial \overline{NE}(Z_{{\mathcal C}^d})$
\cite[II.4.12.2]{kollar1}. By (\ref{qqq}),
$\overline{NE}({X_n})\subseteq \overline{NE}(Z_{{\mathcal C}^d})$
and, since $[D_{d,\alpha}]\in Q({X_n})\subseteq
\overline{NE}({X_n})$, one has that $[D_{d,\alpha}]\in
\partial \overline{NE}({X_n})$. This proves (a).

Consider $d_1,d_2\in S(n)$ and $\alpha_i\in E^{d_i}$, $i=1,2$. If
$d_1=d_2$ and $\alpha_1\not=\alpha_2$, one has clearly that
$[D_{d_1,\alpha_1}]$ and $[D_{d_1,\alpha_2}]$ cannot be proportional
since, in this case, $\Psi^{d_1}(\alpha_1)=\Psi^{d_1}(\alpha_2)$
but, by Proposition \ref{teo1}, the map $\Psi^{d_1}$ is injective.
If $d_1\not=d_2$ and $\{d_1,d_2\}\not=\{2,3\}$ then $\ell({\mathcal
B}^{d_1})\not= \ell({\mathcal B}^{d_2})$ (see Remark
\ref{cardinal}). But, since $N([D_{d_i,\alpha_i}])=\ell({\mathcal
B}^{d_i})$, $i=1,2$, (recall the notation introduced before the
statement of the theorem) one has that $[D_{d_1,\alpha_1}]$ and
$[D_{d_2,\alpha_2}]$ cannot be proportional. Therefore, (b) holds.

If $d\in S(n)\cap \{2,3,4\}$, taking into account the paragraph
after Definition \ref{bbb} and the Adjunction Formula, one has that
$[K_{{X_n}}]\cdot [D_{d,\alpha}]= K_{{X}^{d}}\cdot D_{d,\alpha}=0$
for all $\alpha \in E^d$ (as in the proof of Proposition
\ref{teo1}). This proves (c). Clause (d) follows in a similar
manner.

In particular, from (a), (b) and (c) it follows that  the set of
rays contained in $\partial \overline{NE}({X_n})\cap
\partial Q({X_n})\cap [K_{{X_n}}]^{\perp}$ is not finite. Also, from (a), (b)
and (d) it follows that, if $n\geq 37$, the same happens for the
set of rays contained in $\partial \overline{NE}({X_n})\cap
\partial Q({X_n})\cap [K_{{X_n}}]_{>0}$.

Finally, we shall prove the last assertion of the statement. By the
Hodge index theorem \cite[V.1.9]{har}, the index of the bilinear
pairing $A({X_n})\times A({X_n}) \rightarrow \gr$ induced by the
intersection product is $(1,n)$. Then, taking coordinates in a
certain basis, $Q({X_n})$ can be seen as the half-cone over an
Euclidean ball of dimension $n$, which is strictly convex. This fact
and the inclusion $Q({X_n})\subseteq \overline{NE}({X_n})$ imply
that two non-proportional classes of the form $[D_{d,\alpha}]$ must
belong to different faces of the cone $\overline{NE}({X_n})$.

\end{proof}

\begin{rem}\label{remgordo}
{\rm

Although the statement of Theorem \ref{gordo} shows that, for each
integer $n\geq 37$, there exist infinitely many rays cutting the
region $\partial \overline{NE}(X_n)\cap
\partial Q(X_n)\cap [K_{X_n}]_{>0}$ (which are spanned by elements of
the form $[D_{d,\alpha}]$, with $(d,\alpha)\in E(n)$), a stronger
fact is implicit in that statement (recall also Remark
\ref{cardinal}): if $f_n$ denotes the maximum of the numbers
$N([D_{d,\alpha}])$, with $(d,\alpha)\in E(n)$ (notice that they
depend only on $d$), then $\lim_{n\rightarrow \infty}
f_n=+\infty$. This implies that, for each fixed positive integer
$k$, there exists an integer $n_0>k$ such that, for each $n\geq
n_0$, one can find infinitely many rays $\gr_{>0}z$ cutting the
region $\partial \overline{NE}(X_n)\cap
\partial Q(X_n)\cap [K_{X_n}]_{>0}$ and such that the number of
exceptional divisors $E$ on the blow-up $X_n$ with $z\cdot [E]>0$
is greater than $k$. }
\end{rem}

Now, we shall use the action of the Cremona group in order to extend
further the result given in the above theorem. Let us recall briefly
this action, referring the reader to \cite{duval} and \cite{do-or}
for more details.

Let $X_n$ be as in the statement of Theorem \ref{gordo}. Following
preceding notations, consider $A(X_n)$ as the direct sum
$\mathbb{Z}[L^{\ck_0(n)}]\oplus \mathbb{Z}[E_1^{\ck_0(n)}]\oplus
\cdots \oplus \mathbb{Z}[E_n^{\ck_0(n)}]$. $A(X_n)$, endowed with
the bilinear pairing defined by the intersection product, is a
hyperbolic lattice. Let $Cr_n$ be the subgroup of $Aut(A(X_n))$
generated by the symmetric group $S_n\hookrightarrow Aut(A(X_n))$
(acting on the last $n$ components) and the reflection
$R:A(X_n)\rightarrow A(X_n)$ defined by $R(x):=x+(x\cdot e ) e$,
where
$e:=[L^{\ck_0(n)}]-[E_1^{\ck_0(n)}]-[E_2^{\ck_0(n)}]-[E_3^{\ck_0(n)}]$.
The group $Cr_n$ is called the {\it Cremona group}.

It is obvious that $Cr_n$ acts on the set of (open) half-lines of
$A(X_n)$ with origin at $0$ ({\it rays}). Moreover it also acts on
the set of nef classes in $A(X_n)$ (see \cite{biran}) and, if $D$ is
a nef divisor on $X_n$ such that $D^2=0$ then $[D]\in
\partial \overline{NE}(X_n)$. These facts show that each one of
the infinite rays $\eta\subseteq \partial \overline{NE}(X_n)\cap
\partial Q(X_n)\cap [K_{X_n}]_{\geq 0}$ provided by Theorem \ref{gordo}
gives rise to infinitely many rays contained in the same set: those
belonging to the orbit of $\eta$ by the action of the Cremona group.
Taking this into account, the last result of the paper completes the
one provided by Theorem \ref{gordo}.

\begin{teo}\label{gordo2}

Let $n\geq 37$ be an integer and set $X_n$ and $S(n)$ as in the
statement of Theorem \ref{gordo}. Let ${\mathcal R}$ be the set of
rays in $A(X_n)$ of the form $\mathbb{R}_{>0}[D_{d,\alpha}]$ such
that $d\in S(n)\setminus \{2,3,4\}$ and $\alpha\in E^d$. Then,  for
each ray $\eta\in {\mathcal R}$, the intersection of ${\mathcal R}$
with the orbit of $\eta$ by the action of $Cr_n$ is finite. In
particular, infinitely many of the rays in $\partial
\overline{NE}(X_n)\cap
\partial Q(X_n)\cap [K_{X_n}]_{>0}$ provided in
Theorem \ref{gordo} are generated by elements which belong to
different orbits of the action of the Cremona group.
\end{teo}

\begin{proof}

Given a ray $\eta$ in $A(X_n)$ we shall denote by $\eta'$ to its
{\it primitive} generator, that is, the generator
$a_0[L^{\ck_0(n)}]+\sum_{i=1}^n a_i[E_i^{\ck_0(n)}]$ such that
$\gcd(a_0,a_1,\ldots,a_n)=1$. Since the canonical class is fixed by
the action of $Cr_n$, it is clear that the result follows if we are
able to prove that, for each positive integer $k$, the set
$\{[K_{X_n}]\cdot \eta'\leq k\mid \eta\in {\mathcal R}\}$ is finite.
Let us see this fact.

Choose a positive integer $k$ and, reasoning by contradiction,
assume that the above set is not finite. Let $\eta$ be an arbitrary
ray in ${\mathcal R}$, say generated by
$[D_{d,\alpha}]=m_0[L^{\ck_0(n)}]-\sum_{i=1}^n m_i[E_i^{\ck_0(n)}]$.
Set $\eta'=m'_0[L^{\ck_0(n)}]-\sum_{i=1}^n m'_i[E_i^{\ck_0(n)}]$.
Since the set of possible values of $m'_0$ is unbounded and the set
$$\left\{\frac{[K_{X_n}]\cdot \eta'}{m'_0}\leq \frac{k}{m'_0}\mid \eta\in {\mathcal R}\right\}$$ is not
finite, one has that $\frac{[K_{X_n}]\cdot \eta'}{m'_0}<1$ for some
$\eta\in {\mathcal R}$. But $$\frac{[K_{X_n}]\cdot
\eta'}{m'_0}=\frac{[K_{X_n}]\cdot
[D_{d,\alpha}]}{m_0}=\frac{K_{X^d}\cdot D_{d,\alpha}}{m_0}\geq 1,$$
where the last inequality holds by the proof of Lemma 3 in
\cite{l-n}. This is a contradiction.

\end{proof}

\end{document}